\documentclass{amsart}
\usepackage{amsmath,amsthm,amsopn,amsfonts,fancyhdr,graphicx,epic,amssymb,epsfig,graphics,psfrag}
\numberwithin{equation}{section}
\newtheorem{thm}{Theorem}[section]

\newtheorem{prop}{Proposition}[section]
\theoremstyle{definition}
\newtheorem{defn}{Definition}[section]
\newtheorem{rem}{Remark}[section]

\newtheorem{notation}{Notation}[section]

\newcommand{\Real}{\mathbb R}
\newcommand{\N}{\mathbb N}

\newcommand{\normasemplice}[1]{\|#1\|}

\begin{document}

\title[A nonlinear problem involving isotropic deformations]{Existence of solutions to a nonlinear problem involving isotropic deformations}

\author{Ana Cristina Barroso, Gisella Croce, Ana Margarida Ribeiro}

 \address{Maths Dep. and CMAF, University of Lisbon,
Lisbon, Portugal; LMAH, Le Havre University, Le Havre,
France; Maths Dep. and CMA, FCT-Universidade Nova de Lisboa, 
Caparica, Portugal.}

\email{abarroso@ptmat.fc.ul.pt; gisella.croce@univ-lehavre.fr; amfr@fct.unl.pt}

\thanks{}
% Abstract----------------------------------------------------------------


\begin{abstract}  
In this paper we show, under suitable hypotheses on the boundary datum $\varphi$, 
existence of Lipschitz maps
$u: \Omega \to \Real^2$ satisfying the nonlinear differential inclusion
\[
\left\{
\begin{array}{ll}
D u \in E, \,\, &\mbox{a.e. in } \Omega
\vspace{0.1cm}\\
u=\varphi, &\mbox{on } \partial \Omega \vspace{0.1cm}
\end{array}
\right.
\]
where $\Omega$ is an open bounded subset of $\Real^2$ and $E$ is a
compact subset of $\Real^{2\times 2}$, which is isotropic, that is to say, invariant under orthogonal
transformations.
Our result relies on an abstract existence theorem due to M\"{u}ller and \v{S}ver\'ak which requires
the set $E$ to satisfy a certain in-approximation property. 
\end{abstract}

% Contents----------------------------------------------------------------

% ----------------------------------------------------------------
\maketitle

\section{Introduction}

A microstructure is a structure on a scale between the macroscopic and the atomic ones. 
Microstructures are abundant in nature, for example, they are present in  molecular tissues or in
biomaterials. Crystals such as igneous rocks or metal alloys (for example nickel-aluminium, zinc-lead) also
develop microstructures. 
The microstructure of a material can strongly influence physical properties
such as strength, toughness, ductility, hardness and corrosion resistance. 
This influence can vary as a function of the temperature of the material.

In the last twenty years successful models for studying the behaviour of 
crystal lattices undergoing 
solid-solid phase transitions have been studied. 
In such models it is assumed that the elements of 
crystal lattices have  
certain preferable affine deformations; this is true for example for 
martensite or for quartz crystals (see \cite{BJ, M}).

Denoting by $E$ the set of matrices corresponding to the gradient of these 
deformations, the physical models  motivate
the mathematical question of the existence of solutions to 
Dirichlet problems related to systems of differential inclusions such as
$Du\in E$ a.e. in $\Omega$, $u = \varphi$ on $\partial \Omega$,
where $\Omega$ is a domain of $\Real^n$ and $E\subset \Real^{n\times n}$ is a 
compact set.

Two abstract theories to establish the existence of solutions of general 
differential inclusion problems are due to
Dacorogna and Marcellini (see \cite{DM, DMacta}), whose result is based on  
Baire's category theorem, and M\"{u}ller and \v{S}ver\'ak \cite{MS, MSannals},
who use ideas of convex integration by Gromov \cite{Gromov}.
In these two theories certain convex hulls of the set $E$ play an important role.
We say that a set $E \subseteq \Real^{n \times n}$ is rank one convex if for every
$\xi, \eta \in E$ such that rank$(\xi-\eta) = 1$ and for every $t \in [0,1]$ then
$$t\xi+(1-t)\eta \in E.$$
In the spirit of the usual definition of the convex hull, the rank one convex hull of a set 
$E$ can be defined as the smallest rank one convex set that
contains $E$. However, in order to solve differential inclusion problems, several authors, namely
M\"{u}ller and \v{S}ver\'ak \cite{MS, MSannals}, consider the following alternative notion of the 
rank one convex hull of a compact set $E\subset\mathbb{R}^{N\times n}$, denoted by $E^{rc}$:
\begin{eqnarray*}
& & E^{rc} =\left\{\xi\in\mathbb{R}^{N\times n}:f(\xi) \leq 0, \text{ for every }\right. \\
& & \left. \hspace{1,5cm} \text{rank one convex function }f\in\mathcal{F}^{E}\right\}
\end{eqnarray*} 
where
$$\mathcal{F}^{E} =\left\{ f:\mathbb{R}^{N\times n}\rightarrow
\mathbb{R}: f\lfloor _{E}\leq0\right\}.$$ 

In each of the aforementioned theories, provided certain approximation properties 
hold, if the gradient of the 
boundary datum $\varphi$ belongs to the interior of the appropriate convex hull of $E$, then 
there exists a solution 
$u\in \varphi+W^{1,\infty}_0(\Omega,\Real^n)$ to $Du\in E$ a.e. in $\Omega$.

Using these abstract theorems various interesting problems related to the existence of microstructures have 
been solved, such as the two-well
problem, where $E=\mathcal{S}\mathcal{O}(2)A \cup \mathcal{S}\mathcal{O}(2)B$, where $A$ and $B$ are two 
fixed $\Real^{2\times 2}$ matrices and  $\mathcal{S}\mathcal{O}(2)$ stands for the special orthogonal group 
(see \cite{DMtwowell, DM, dolzmann, M, MS}).

In this article we study the case where the set $E$ is an arbitrary $\Real^{2\times 2}$  isotropic set, 
that is, invariant under orthogonal transformations. More precisely,
we assume that $E$ is a compact subset 
of $\Real^{2\times 2}$ such that 
$RES\subseteq E$ for every $R,S$ in the orthogonal group $\mathcal{O}(2)$. 
Let $\Omega$ be an open bounded subset of $\Real^2$. We investigate the existence of weakly 
differentiable maps $u:\Omega \to \Real^2$ that satisfy
\begin{equation}\label{problema di partenza}
\left\{
\begin{array}{ll}
D u \in E, \,\, &\mbox{a.e. in } \Omega
\\
u=\varphi, \,\, &\mbox{on } \partial \Omega.
\end{array}
\right. 
\end{equation}
Note that this problem is fully nonlinear. 
Since $E$ is isotropic, it can be written as
\begin{equation}
E=\{\xi \in \Real^{2\times 2}: (\lambda_1(\xi), \lambda_2(\xi))\in K
\}\,,
\end{equation}
for some compact set 
$K\subset \{(x,y)\in \Real^{2}: 0\leq x\leq y\}$, 
where we have denoted by $0 \leq \lambda_1(\xi)\leq \lambda_2(\xi)$
the singular values of the matrix $\xi$, that is, the eigenvalues
of the matrix $\sqrt{\xi \xi^t}$.

Thanks to the properties of the singular values (see Section 3), 
problem (\ref{problema di partenza}) can be rewritten in the
following equivalent way: for almost every $x \in \Omega$ there exists
$(a,b) \in K$ such that
\[
\left\{
\begin{array}{l}
\normasemplice{D u(x)}^2{}=a^2+b^2,
\vspace{0.1cm}\\
|\text{det}\,D u(x)|=ab,
\end{array} \right.
\]
and $u(x)=\varphi(x)$ if $x \in \partial \Omega$.
In the case where $K$ consists of a unique point these two equations are the vectorial eikonal 
equation and the equation of prescribed absolute value of the Jacobian determinant.

The main result of our article is the following
\begin{thm}\label{teorema d'esistenza nel nostro caso sverakmuller} 
Let $E:=\{\xi \in \Real^{2\times 2}: (\lambda_1(\xi),\lambda_2(\xi))\in K \}$ 
where
$ K\subset\{(x,y) \in \Real^2: 0<x \leq y\}$ is a compact set.
Let $\Omega\subset \Real^2$ be a bounded open set and let 
$\varphi \in C^1_{piec}(\overline{\Omega}, \Real^2)$ be such that 
$D \varphi \in E \cup \textnormal{int}E^{rc}$ a.e. in $\Omega$. Then there exists
a map $u\in \varphi+ W^{1,\infty}_0(\Omega, \Real^2)$ such that
$Du \in E$ a.e. in $\Omega$.
\end{thm}
To prove this theorem we use the following characterization of $E^{rc}$ which can be obtained using results of
Cardaliaguet and Tahraoui \cite{CT} on isotropic sets (see Section 3 for more details):
\begin{eqnarray*}
& & E^{rc} = \Big\{\xi \in \Real^{2\times 2} : 
f_{\theta}(\lambda_1(\xi),\lambda_2(\xi)) 
\\
& & \hspace{2cm}\leq  \left. \max\limits_{(a,b) \in K}f_{\theta}(a,b), 
\,\forall\,\,\theta \in [0,\max\limits_{(a,b)\in K}b]\right\}
\end{eqnarray*}
where
$$
f_{\theta}(x,y) := xy + \theta(y-x), \, 0 < x \leq y, \, \theta \geq 0.
$$
This characterization is only known to
hold in dimension $2\times 2$ and this is the reason our analysis is restricted to this case.  

The above existence theorem was first 
obtained by Croce in \cite{C}
using the theory developed by Dacorogna and Marcellini and a refinement due to Dacorogna and Pisante \cite{DP}.
In this article we treat the same problem using the theory by M\"{u}ller and 
\v{S}ver\'ak, which leads to different technical difficulties, nevertheless we arrive at the same result.
Notice that our only restriction on the compact, isotropic set $E$ is that it contains no singular matrices. 
The reason for this is to be able to construct an in-approximation sequence for $E$ (see Section 4).
This restriction was already present in \cite{C}.
We point out that in the case where 
$K$ consists of a unique point and $K\subset \Real^n$, $n\geq 2$ 
the same existence result was obtained by Dacorogna and Marcellini in \cite{DM}. 

\section{Notions of Convexity}

In this section we gather together some generalized convexity notions and properties 
which will be useful in the sequel. For more details on this matter we refer to
\cite{D} and \cite{DR}.

\begin{notation} For $\xi\in\mathbb{R}^{N\times n}$ we let
\[
T\left(  \xi\right)  =\left(
\xi,\mathrm{adj}_{2}\xi,\ldots,\mathrm{adj}_{N\wedge n}\xi\right)
\in\mathbb{R}^{\tau(N,n)}
\]
where $\mathrm{adj}_{s}\xi$ stands for the matrix of all $s\times
s$ subdeterminants of the matrix $\xi,$ $1\leq s\leq N\wedge
n=\min\left\{  N,n\right\}  $ and where
\[
\tau=\tau\left(  N,n\right)  =\underset{s=1}{\overset{m\wedge
n}{\sum}}\binom {N}{s}\binom{n}{s}
 \text{ and }\binom{N}{s}
 =\frac{N!}{s!\left(  N-s\right)  !}.
\]
In particular, if $N=n=2,$ then $T\left(  \xi\right)  =\left(
\xi,\det \xi\right).$
\end{notation}

\begin{defn}\label{convfunctions}
\begin{itemize}
\item[(i)] A function 
$f:\mathbb{R}^{N\times n}\rightarrow\mathbb{R}\cup\left\{  +\infty\right\}$ 
is said to be \emph{polyconvex} if there exists a convex function
$g:\mathbb{R}^{\tau(N,n)}\longrightarrow
\mathbb{R}\cup\left\{+\infty\right\}$ such that
$$f(\xi)=g(T(\xi)).$$
\item[(ii)] A Borel measurable function 
$f:\mathbb{R}^{N\times n}\rightarrow \mathbb{R}$ is said to be 
\emph{quasiconvex} if
\[
f\left( \xi\right)\,\mathrm{\operatorname*{meas}}(U)\leq
\int_{U}f\left(  \xi+D\varphi\left(x\right)  \right)  dx
\]
for every bounded open set $U\subset\mathbb{R}^{n}$,
$\xi\in\mathbb{R}^{N\times n}$ and 
$\varphi\in W_{0}^{1,\infty}\left(  U;\mathbb{R}^{N}\right)$.
\item[(iii)] A function 
$f:\mathbb{R}^{N\times n}\rightarrow\mathbb{R}\cup\left\{  +\infty\right\}$ 
is said to be \emph{rank one convex} if
\[
f\left( t\xi+(1-t)\eta\right)  \leq 
t\,f\left(\xi\right) +\left(  1-t\right) \,f\left( \eta\right)
\]
for every $t\in\left[  0,1\right] $ and every
$\xi,\eta\in\mathbb{R}^{N\times n}$ with
$\operatorname*{rank}(\xi-\eta)  =1$.
\end{itemize}
\end{defn}

It is well known that, if $f : \mathbb{R}^{N\times n}\rightarrow\mathbb{R}$, then
$$f \, \textnormal{polyconvex} \Rightarrow f \, \textnormal{quasiconvex} \Rightarrow
f \, \textnormal{rank one convex}.$$

\begin{defn}\label{definitiongeneralizedconvexities}
\begin{itemize}
\item[(i)] We say that $E\subset\mathbb{R}^{N\times n}$ is
\emph{polyconvex} if 
$$\left.\begin{array}{c}\vspace{0.2cm}\displaystyle{\sum_{i=1}^{\tau+1}t_i 
T(\xi_i)=T\left(\sum_{i=1}^{\tau+1}t_i \xi_i\right)}\\
\xi_i\in E,\ t_i\ge 0,\ \displaystyle{\sum_{i=1}^{\tau+1}t_i=1}
\end{array}\right\}\Rightarrow \sum_{i=1}^{\tau+1}t_i \xi_i\in
E.$$ 
In the case $N=n=2$ we recall that $\tau+1=5$ and $E$ is polyconvex if
for all $t_i \geq 0$ with
$\displaystyle \sum_{i=1}^{5}t_i = 1$ and for all $\xi_i \in E$ with
$$\sum_{i=1}^5t_i\det \xi_i = \det\left(\sum_{i=1}^5t_i\xi_i\right)$$
then $\displaystyle \sum_{i=1}^5t_i\xi_i \in E$.
\item[(ii)] Let $E\subset\mathbb{R}^{N\times n}$. We say that $E$ is
\emph{rank one convex} if for every $\xi,\eta\in E$ such that
$\operatorname*{rank}(\xi-\eta)=1$ and for every $t \in [0,1]$ then
$$t \xi+(1-t)\eta\in E.$$
\end{itemize}
\end{defn}

In the following theorem we mention some properties of polyconvex and rank one
convex sets that can be found in the literature. Properties (i) and (iii) were 
proved in \cite{DR}, see also \cite{D}, whereas (iv) is well 
known. The characterization given in (ii) corresponds to the definition of 
polyconvex set used by Cardaliaguet and Tahraoui in \cite{CT}. 

\begin{thm}\label{propertiesconvex}
Let $E\subset\mathbb{R}^{N\times n}$.
\begin{itemize}
\item[(i)] The set $E$ is polyconvex if and only if 
$$E=\{\xi\in\mathbb{R}^{N\times n}: T(\xi)\in \operatorname*{co}T(E)\},$$ 
where $\operatorname*{co}T(E)$ denotes the convex hull of $T(E)$.
\item[(ii)] If $E$ is compact, then it is polyconvex if and only if there exists a 
polyconvex function $f:\mathbb{R}^{N\times n}\rightarrow\mathbb{R}$ such that 
$$E=\{\xi\in\mathbb{R}^{N\times n}:\ f(\xi)\le 0\}.$$
\item[(iii)] If $E$ is polyconvex (respectively, rank one convex) then 
$\mathrm{int}\,E$ is also polyconvex (respectively, rank one convex). However, even 
if $N=n=2$, $\overline{E}$ is not necessarily polyconvex (respectively, rank one 
convex).
\item[(iv)] If $E$ is polyconvex then it is rank one convex. 
\end{itemize}
\end{thm}

\begin{rem}
{\rm The converse of (iv) is, in general, 
false. }
\end{rem}

The concepts of polyconvexity and rank one convexity for sets were introduced as a tool for
solving differential inclusion problems through the notion of convex hull in these generalized senses.
This lead to different definitions of these hulls that can be found in the literature. The ones we give in 
Definition~\ref{DMhulls} are the natural ones in the spirit of the classical notion of convex hull. 
These hulls were considered by Dacorogna and Marcellini to establish their abstract existence theorem 
for differential inclusions of the type we are considering. This was the theory used by Croce in \cite{C}.
In Definition~\ref{MShulls} we recall the notions of polyconvex and rank one convex hulls of a given set
used by M\"uller and {\v{S}}ver{\'a}k \cite{MSannals} in their convex integration method. These are the 
notions we will use in our existence result in Sections 4 and 5.

\begin{defn}\label{DMhulls}
The \emph{polyconvex} and the \emph{rank one convex} hull of a set
$E\subset\mathbb{R}^{N\times n}$ are, respectively, the smallest
polyconvex and rank one convex sets containing $E$ and are respectively denoted by
$\operatorname*{Pco}E$ and $\operatorname*{Rco}E$.
\end{defn}

From Theorem \ref{propertiesconvex}, the following inclusions hold
$$E\subset \operatorname*{Rco}E\subset \operatorname*{Pco}E\subset 
\operatorname*{co}E,$$ where $\operatorname*{co}E$ denotes the convex hull of $E$.
 
It was proved by Dacorogna and Marcellini in \cite{DM} that 
\begin{eqnarray}
& & \operatorname*{Pco}E =\left\{\xi\in\mathbb{R}^{N\times n}:\,
\displaystyle T(\xi)=\sum_{i=1}^{\tau+1}t_i T(\xi_i), \right. \nonumber \\ 
\label{pco} 
& & \hspace{2,5cm} \left.\xi_i\in E, \ t_i\ge 0, \ \sum_{i=1}^{\tau+1}t_i=1\right\}. 
\end{eqnarray}

One has (see \cite{DR}) that $\operatorname*{Pco}E$ and 
$\operatorname*{Rco}E$ are open if $E$ is open, and $\operatorname*{Pco}E$ is 
compact if $E$ is compact. However, it isn't true that $\operatorname*{Rco}E$ is 
compact if $E$ is compact (see \cite{kolar}).
 
It is well known that, for $E\subset\mathbb{R}^{N\times n}$,
\begin{eqnarray}\label{coE}
& & \operatorname*{co}E =\left\{\xi\in\mathbb{R}^{N\times n}: f(\xi) \leq0, 
\right.\\ \nonumber
& & \left.\hspace{2cm}  \text{for every convex function }
f\in{\mathcal{F}}^{E}_\infty\right\} \\ \label{closurecoE} 
& & \overline{\operatorname*{co}E}  =\left\{ \xi\in\mathbb{R}^{N\times n}: 
f(\xi)  \leq0, \right. \\ \nonumber
& & \left. \hspace{2cm} \text{for every convex function }f\in\mathcal{F}^{E}\right\}
\end{eqnarray}
where $\overline{\operatorname*{co}E}$ denotes the closure of the
convex hull of $E$ and  
\begin{align*}
{\mathcal{F}}^{E}_\infty  &  =\left\{  f:\mathbb{R}^{N\times n}
\rightarrow\mathbb{R}\cup\left\{ +\infty\right\}
:\left.  f\right\vert _{E}\leq0\right\} \\
\mathcal{F}^{E}  &  =\left\{  f:\mathbb{R}^{N\times n}\rightarrow
\mathbb{R}:\left.  f\right\vert _{E}\leq0\right\}  .
\end{align*}
  
Analogous representations to $(\ref{coE})$ can be obtained in the
polyconvex and rank one convex cases: 
\begin{eqnarray*}
& & \operatorname*{Pco}E  = \left\{ \xi\in\mathbb{R}^{N\times n}:
f(\xi) \leq0, \right. \\ 
& & \left. \hspace{1,6cm} \text{for every polyconvex function }
f\in{\mathcal{F}}^{E}_\infty\right\}, \\
& & \operatorname*{Rco}E =\left\{ \xi\in\mathbb{R}^{N\times n}:
f(\xi) \leq0, \right. \\
& & \left. \hspace{0,8cm}\text{for every rank one convex function }
f\in{\mathcal{F}}^{E}_\infty\right\}.
\end{eqnarray*}
However,
$(\ref{closurecoE})$ can only be generalized to the polyconvex
case if the sets are compact, in the rank one convex case
$(\ref{closurecoE})$ is not true, even if compact sets are
considered (see (\ref{phulls}) and (\ref{rhulls})).
This shows that the hulls considered in Definitions~\ref{DMhulls} and \ref{MShulls}
are, in fact, different.

\begin{defn}\label{MShulls}
If $E\subset\mathbb{R}^{N\times n}$ is a compact set,
\begin{eqnarray*}
& & E^{pc} =\left\{\xi\in\mathbb{R}^{N\times n}:f(\xi) \leq 0, \right. \\
& & \left. \hspace{1,7cm} \text{for every polyconvex function }f\in\mathcal{F}^{E}\right\}, \\
& & E^{rc} =\left\{\xi\in\mathbb{R}^{N\times n}:f(\xi) \leq 0, \right. \\
& & \left. \hspace{0,8cm} \text{for every rank one convex function }f\in\mathcal{F}^{E}\right\}.
\end{eqnarray*} 
If $U\subset\mathbb{R}^{N\times n}$ is an open set,
\begin{align*}
U^{pc} &  =\bigcup_{\substack{E\subset U\\E\text{ compact}}} E^{pc}  \\
U^{rc}  &  =\bigcup_{\substack{E\subset U\\E\text{ compact}}} E^{rc}.
\end{align*} 
\end{defn}

\begin{rem}\label{remarkshulls}
{\rm Clearly  $E^{rc} \subseteq E^{pc}$ and $U^{rc} \subseteq U^{pc}$.
Notice that $E^{pc}$ and $E^{rc}$ are closed sets and $U^{pc}$ and $U^{rc}$ are open sets.
Also, for compact sets $E$, these hulls coincide with what is 
denoted by $\operatorname*{Pco}\nolimits_{f}E$ and 
$\operatorname*{Rco}\nolimits_{f}E$ in \cite{D, DR}.

Both in the open and the compact cases the above sets are, respectively, polyconvex 
and rank one convex.}
\end{rem}

We next point out the relations between the closures of the convex hulls
and the sets introduced in the above definition.

If $E$ is a compact set, then
\begin{equation}\label{phulls}
\operatorname*{Pco}E=\overline{\operatorname*{Pco}E}=E^{pc},
\end{equation} 
but, in general,
\begin{equation}\label{rhulls}
\operatorname*{Rco}E\subsetneq\overline{\operatorname*{Rco}E}
\subsetneq E^{rc}.
\end{equation}
However, in some cases these sets coincide (see the following section for 
more details).

If $U$ is an open set, then
\begin{equation}\label{openphulls}
U^{pc}=\operatorname*{Pco}U
= \bigcup_{\substack{E\subset U\\E\text{ compact}}} \operatorname*{Pco}E
\subset\overline{\operatorname*{Pco}U}
\end{equation}
but, in general,
$$U^{rc} \supsetneq \operatorname*{Rco}U  = 
\bigcup_{\substack{E\subset U\\E\text{ compact}}} \operatorname*{Rco}E.$$

\section{Properties of Isotropic Sets}

Our aim in this article is to study a differential inclusion problem involving
a compact set $E$ which is isotropic, that is, invariant
under orthogonal transformations. Therefore in this section we mention  
some results on isotropic subsets of $\Real^{n \times n}$. However 
Theorems \ref{Erc=Epc} and \ref{caracEpc} are established for $n = 2$, since they rely on results of
Cardaliaguet and Tahraoui which are known only in dimension $2\times 2$. These results are crucial to
obtain our existence theorem which holds, thus, only in $\Real^{2\times 2}$.

\begin{defn}
Let $E$ be a subset of $\Real^{n\times n}$. We say $E$ is isotropic if 
$RES\subseteq E$ for every $R, S$ in the orthogonal group 
$\mathcal{O}(n)$. 
\end{defn}

We denote by $0 \leq \lambda_1(\xi)\leq \cdots \leq \lambda_n(\xi)$
the singular values of the matrix $\xi$, that is, the eigenvalues
of the matrix $\sqrt{\xi \xi^t}$. 

We recall the following properties on the singular values:
$$
\begin{array}{c}
\displaystyle \prod_{i=1}^n\lambda_i(\xi)=|\det \xi|
\\
\displaystyle \sum_{i=1}^n(\lambda_i(\xi))^2=\normasemplice{\xi}^2.
\end{array}
$$
From these properties it follows that, in the $2 \times 2$ case,
$\lambda_1$ and $\lambda_2$ are given by
$$
\begin{array}{l}
\displaystyle \lambda_1(\xi)=\frac 12
\left[\sqrt{\normasemplice{\xi}^2+2
|\det \xi|}-\sqrt{\normasemplice{\xi}^2-2 |\det \xi|}\right]
\vspace{0.15cm}
\\
\displaystyle \lambda_2(\xi)=\frac 12
\left[\sqrt{\normasemplice{\xi}^2+2
|\det \xi|}+\sqrt{\normasemplice{\xi}^2-2 |\det \xi|}\right].
\end{array}
$$ 
In addition, for every $\xi \in \Real^{n \times n}$, $R, S \in \mathcal{O}(n)$
$$\lambda_i(\xi) = \lambda_i(R\xi S),$$ 
$\lambda_i$ are continuous functions, 
$\displaystyle \prod_{i=k}^n\lambda_i(\xi)$ is 
polyconvex for any $1 \leq k \leq n$ and $\lambda_n$ is a norm.
Moreover the following decomposition holds (see \cite{HJ}): for every matrix $\xi$ 
there exist $R,S \in \mathcal{O}(n)$ such that
\begin{eqnarray*}
\xi&=&R\left(\begin{array}{ccc}
\lambda_1(\xi) &  & \\
 & \ddots & \\
 & & \lambda_n(\xi)
\end{array}\right)S \\
&=& R \, {\rm diag}(\lambda_1(\xi), \cdots, \lambda_n(\xi)) S.
\end{eqnarray*}
Due to these properties, any isotropic set $E$ may be written in the form 
\begin{equation}\label{isotropic}
\{\xi \in \Real^{n\times n}: (\lambda_1(\xi),\cdots, \lambda_n(\xi))\in \Gamma\}
\end{equation}
where $\Gamma$ is a set contained in 
$\{(x_1, \cdots,x_n)\in \Real^n: 0\leq x_1 \leq \cdots \leq x_n\}$. 
Clearly if $\Gamma$ is compact (respectively, open) then $E$ is also compact
(respectively, open).
On the other hand, if $E$ is compact the set $\Gamma$ can be taken to be compact and if 
$E$ is open (\ref{isotropic}) holds for an open set $\Gamma \subset \Real^n$. 

In the following theorems we establish some results on isotropic sets.

\begin{thm}\label{Ercisotropic}
\begin{itemize}
\item[(i)] If $E \subseteq \Real^{n\times n}$ is isotropic then $\operatorname*{Pco}E$ is isotropic.
\item[(ii)] If $U \subseteq \Real^{n\times n}$ is open, bounded and isotropic then
$$U^{pc} = \bigcup_{\substack{E\subset U\\E\text{ compact}
\\E\text{ isotropic}}} E^{pc}$$
and
$$U^{rc} = \bigcup_{\substack{E\subset U\\E\text{ compact}
\\E\text{ isotropic}}} E^{rc}.$$
\item[(iii)] If $E \subseteq \Real^{n\times n}$ is compact and isotropic then $E^{rc}$ 
is compact and isotropic.
\end{itemize}
\end{thm}
\begin{proof}
(i) Let $\xi$ be a matrix belonging to $\operatorname*{Pco}E$ and let
$R,S$ be two orthogonal matrices. By (\ref{pco}), and using the fact that $E$ is isotropic,
it is easy to see that
$R\xi S \in \operatorname*{Pco}E$ so $\operatorname*{Pco}E$ is isotropic.

(ii) To prove the first equality it suffices to show that 
$$\bigcup_{\substack{E\subset U\\E\text{ compact}}} E^{pc}
\subseteq \bigcup_{\substack{E\subset U\\E\text{ compact}
\\E\text{ isotropic}}} E^{pc},$$
since the reverse inclusion is clear.

Let $E\subset U$ be a compact set. Then 
$\displaystyle \bigcup_{\substack{R, S \in \mathcal{O}(n)}}RES$ 
is an isotropic, bounded subset of $U$. It is also easy to see that this set is
closed. Therefore we conclude that
\begin{eqnarray*}
\bigcup_{\substack{E\subset U\\E\text{ compact}}} E^{pc}
&\subseteq& \bigcup_{\substack{E\subset U\\E\text{ compact}}} 
\left(\bigcup_{\substack{R, S \in \mathcal{O}(n)}}RES\right)^{pc} \\
&\subseteq& \bigcup_{\substack{V\subset U\\V\text{ compact}
\\V\text{ isotropic}}} V^{pc}.
\end{eqnarray*} 

A similar argument proves the second equality.

(iii) As $E$ is compact, $\operatorname*{Pco}E$ is also compact and, by definition, $E^{rc}$ is closed. 
Since by Remark~\ref{remarkshulls} and (\ref{phulls}) we have
$$E^{rc} \subseteq E^{pc} = \operatorname*{Pco}E,$$
we conclude that $E^{rc}$ is bounded and hence compact.

To prove that $E^{rc}$ is isotropic we will show that if
$\xi\notin E^{rc}$ then for every 
$R,S\in \mathcal{O}(n)$ one has $R\xi S\notin E^{rc}$.
Let $\xi \notin E^{rc}$, then there exists a rank one convex function 
$f:\Real^{n\times n}\to \Real$ such that
$f|_{E}\leq 0$ and $f(\xi)>0.$
Let $R,S \in \mathcal{O}(n)$ and define
$$f_1(\eta)=f(R^{-1}\eta S^{-1}).$$ 
Then $f_1$ is rank one convex
and for all $\eta \in E$,
$$f_1(\eta)=f(R^{-1}\eta S^{-1})\leq 0,$$
as $R^{-1}\eta S^{-1} \in E$. However $f_1(R\xi S)=f(\xi)>0$ and so
$R\xi S$ doesn't belong to $E^{rc}$.
\end{proof}

For the purposes of our existence theorem we will restrict our attention to
compact, isotropic sets $E$ of the form
\begin{equation}\label{notreE}
E=\left \{\xi \in \Real^{2\times 2}: (\lambda_1(\xi), \lambda_2(\xi))\in K \right\}
\end{equation}
for which 
\begin{equation}\label{Knumerofinitodipunti}
K \subset \left\{(x,y) \in \Real^2: 0<x \leq y\right\}.
\end{equation}

The proof of the next theorem relies on results of \cite{CT} and \cite{C}.

\begin{thm}\label{Erc=Epc}
\begin{itemize}
\item[(i)] If $E$ is a compact, isotropic set of the form 
(\ref{notreE}) and (\ref{Knumerofinitodipunti})
then $$E^{rc} = E^{pc} = \operatorname*{Pco}E.$$
\item[(ii)] If $U$ is an open, isotropic and bounded subset of $\Real^{2\times 2}$ 
then $$U^{rc} = U^{pc} = \operatorname*{Pco}U$$
and $$(U^{rc})^{rc} = U^{rc}.$$
\end{itemize}
\end{thm}
\begin{proof}
(i) It was shown by \cite[Theorem 5.1]{CT} that any compact, isotropic, rank one
convex subset of $\Real^{2\times 2}$ is polyconvex. This, together with 
Theorem~\ref{Ercisotropic} (iii) and
Remark~\ref{remarkshulls}, shows that if $E$ is a compact, isotropic subset of 
$\Real^{2\times 2}$ then 
$$E^{rc} = E^{pc} = \operatorname*{Pco}E.$$ 

(ii) By equation (\ref{openphulls}), Theorem~\ref{Ercisotropic} (ii) and part (i) we have
\begin{equation}\label{eqn}
\operatorname*{Pco}U = U^{pc} = \bigcup_{\substack{E\subset U\\E\text{ compact}
\\E\text{ isotropic}}} E^{pc} = \bigcup_{\substack{E\subset U\\E\text{ compact}
\\E\text{ isotropic}}} E^{rc} = U^{rc}.
\end{equation}
Using this result we obtain
$$(U^{rc})^{rc} = (U^{pc})^{rc}.$$
By Remark~\ref{remarkshulls} $U^{pc}$ is open. Moreover, by Theorem~\ref{Ercisotropic} (i) and 
equation~(\ref{openphulls}) $U^{pc}$ is isotropic.
Thus we may use equation (\ref{eqn}) once again 
to conclude that
$$(U^{pc})^{rc} = (U^{pc})^{pc} = \operatorname*{Pco}(\operatorname*{Pco}U) =
\operatorname*{Pco}U = U^{rc}.$$
Thus 
$$(U^{rc})^{rc} = U^{rc}.$$
\end{proof}

Letting
\begin{equation}\label{ftheta}
f_{\theta}(x,y) := xy + \theta(y-x), \, 0 < x \leq y, \, \theta \geq 0
\end{equation}
the following result is a consequence of Proposition 6.7 in \cite{CT}, 
Propositions 2.4 and 2.5 in \cite{C} and (\ref{phulls}).

\begin{thm}\label{caracEpc}
Let $E$ be a compact, isotropic set of the form (\ref{notreE}) and 
(\ref{Knumerofinitodipunti}).
Then
\begin{eqnarray*}
& & E^{pc} = \operatorname*{Pco}E = \Big\{\xi \in \Real^{2\times 2} : 
f_{\theta}(\lambda_1(\xi),\lambda_2(\xi)) 
\\
& & \hspace{2cm}\leq  \left. \max\limits_{(a,b) \in K}f_{\theta}(a,b), 
\,\forall\,\,\theta \in [0,\max\limits_{(a,b)\in K}b]\right\},
\\
& & {\rm int}E^{pc} = {\rm int}\operatorname*{Pco}E = \Big\{\xi \in 
\Real^{2\times 2} : f_{\theta}(\lambda_1(\xi),\lambda_2(\xi)) 
\\
& & \hspace{2cm} <  \left. \max\limits_{(a,b) \in K}f_{\theta}(a,b), 
\,\forall\,\,\theta \in [0,\max\limits_{(a,b)\in K}b]\right\}.
\end{eqnarray*}
\end{thm}

\begin{rem}\label{Ksingularset}
{\rm (i) Taking into account Theorem~\ref{Erc=Epc} (i) it follows that
if $E$ is a compact, isotropic set of the form (\ref{notreE}) and 
(\ref{Knumerofinitodipunti}) then 
\begin{eqnarray*}
& & E^{rc} = \Big\{\xi \in \Real^{2\times 2} : 
f_{\theta}(\lambda_1(\xi),\lambda_2(\xi)) 
\\
& & \hspace{2cm}\leq  \left. \max\limits_{(a,b) \in K}f_{\theta}(a,b), 
\,\forall\,\,\theta \in [0,\max\limits_{(a,b)\in K}b]\right\},
\\
& & {\rm int}E^{rc} = \Big\{\xi \in 
\Real^{2\times 2} : f_{\theta}(\lambda_1(\xi),\lambda_2(\xi)) 
\\
& & \hspace{2cm} <  \left. \max\limits_{(a,b) \in K}f_{\theta}(a,b), 
\,\forall\,\,\theta \in [0,\max\limits_{(a,b)\in K}b]\right\}.
\end{eqnarray*}

(ii) It was shown by Croce in \cite[Theorem 3.1]{C} that the set representing 
$E^{pc}$ also coincides with $\operatorname*{Rco}\,E$. Therefore, in this case,
$E^{rc} = \operatorname*{Rco}\,E$ but our existence result is independent of this fact.
We also point out that this characterization of $\operatorname*{Rco}\,E$ does not follow from
Theorem 5.1 in \cite{CT} since it is not known, a priori, that $\operatorname*{Rco}\,E$ is compact.

(iii) In the particular case where $K$ is composed of a unique point $(a,b)$ we obtain
$$
E^{pc} = E^{rc} = \left\{\xi \in \Real^{2\times 2} : \lambda_1(\xi)\cdot\lambda_2(\xi) 
\leq ab, \lambda_2(\xi) \leq b \right\}.
$$
This result was first obtained by Dacorogna and Marcellini in \cite{DM} who also showed
that this set coincides with $\operatorname*{Rco}\,E$ in
the more general framework of $\Real^{n\times n}$ matrices.}
\end{rem}

\section{In-Approximation} 
To show  Theorem~\ref{teorema d'esistenza nel nostro caso sverakmuller} we 
will use an existence  result due to M\"uller and \v{S}ver\'ak \cite{MS} which
requires the following in-approximation property.

\begin{defn}{\it{(In-approximation)}}\label{inapprox}
Let $E$ be a compact subset of $\Real^{N\times n}$. We say that a sequence of  
open sets $U_i \subseteq \Real^{N\times n}$ 
is an in-approximation of $E$ if 
\begin{enumerate}
\item $U_i\subseteq U_{i+1}^{rc}$; 
\item $\sup \limits_{\xi \in U_i}\textnormal{dist}(\xi,E) \to 0 
\; \; {\rm as} \; \; i \to + \infty$.
\end{enumerate}
\end{defn}

In this section we will show that the set $E$, defined by (\ref{notreE}) and
(\ref{Knumerofinitodipunti}), admits an in-approximation. 

\begin{defn}\label{defcasosoloab} Let
$0 < \varepsilon_n < \frac{a_0}{2}$, where $a_0 = \min\limits_{(a,b) \in K} a$,
be a decreasing sequence 
such that $\varepsilon_n \to 0^+$. 
For $(a,b) \in K$ we define the open sets,
\[
\begin{array}{l}
R^n_{(a,b)}=
\{(x,y)\in \Real^2: a-2\varepsilon_{n}< x < a-\varepsilon_{n},
\\
 \phantom {R^n_{(a,b)}= \{ (x,y)\in T:} \,\,\displaystyle
b-\varepsilon_n < y < b -\frac{\varepsilon_{n}}{2}\}
\end{array}
\]
and    
$\displaystyle 
U_n:=\left\{\xi \in \Real^{2\times 2}: (\lambda_1(\xi),\lambda_2(\xi))\in
O_n\right\}$, where 
$$\displaystyle O_n=\bigcup\limits_{(a,b)\in K}R^n_{(a,b)}.$$
\end{defn}

\begin{rem}{\rm Notice that, given the choice of $\varepsilon_n$, the sets 
$R^n_{(a,b)}$ lie in the 
region $\{(x,y) \in \Real^2 : 0 < x \leq y\}$. Also, the set $U_n$ is open since the functions
$\lambda_1$ and $\lambda_2$ are continuous.}
\end{rem}

\begin{thm}\label{inapproximationteoexistence}
Let $E$ be given by (\ref{notreE}) and (\ref{Knumerofinitodipunti}). Then $E$
admits an in-approximation.
\end{thm}
\begin{proof}
We will show that the sequence $U_n$ of the above definition is an in-approximation 
for $E$.

\noindent {\it{Step 1.}} We begin by showing the first 
condition of the definition of in-approximation, that is,
$U_n \subseteq U_{n+1}^{rc}.$
To this end it suffices to show that, for every $n$,
\begin{equation}\label{inclusionefine}
V^{n}_{(a,b)} \subseteq  (V^{n+1}_{(a,b)})^{rc},
\end{equation}
where 
$$V^{n}_{(a,b)} = \left\{\xi \in \Real^{2\times 2}: 
(\lambda_1(\xi),\lambda_2(\xi))\in R^n_{(a,b)}\right\}.$$

Let $\xi \in V^{n}_{(a,b)}.$ 
We will show there exists a compact set $C\subset V^{n+1}_{(a,b)}$ 
such that $\xi \in C^{rc}$.
Since the sequence $\varepsilon_n$ is decreasing and $\xi \in V^{n}_{(a,b)}$
we may choose $0 < \delta < \varepsilon_{n+1}$ such that 
$\lambda_1(\xi) \leq a - \varepsilon_{n+1} - \delta$
and $\lambda_2(\xi) \leq b - \frac{\varepsilon_{n+1}}{2} - \frac{\delta}{2}$.
Letting
\begin{eqnarray*}
C=\left\{\eta \in \Real^{2\times 2}: 
\lambda_1(\eta)=a-\varepsilon_{n+1}-\delta, \right.
\\
\left. \hspace{1cm}\lambda_2(\eta)=b-\frac{\varepsilon_{n+1}}{2}-\frac{\delta}{2}
\right\}
\end{eqnarray*}
by the choice of $\delta$ and Remark~\ref{Ksingularset} (iii)
it follows that $C$ is a non-empty compact subset of $V^{n+1}_{(a,b)}$
and $\xi \in C^{rc}$.

\noindent {\it{Step 2.}} We now proceed with the proof of the second property 
of the in-approximation.
Let $\xi \in U_n$. Then there exists $(a,b) \in K$ such that 
$(\lambda_1(\xi), \lambda_2(\xi)) \in R^n_{(a,b)}$. Let 
$\eta = \left(\begin{array}{cc}
a & 0\\
0 & b 
\end{array}\right)$
and let $R$, $S$ be orthogonal matrices such that
$R\xi S = \left(\begin{array}{cc}
\lambda_1(\xi) & 0\\
0 & \lambda_2(\xi) 
\end{array}\right)$. Then $R^{-1}\eta S^{-1} \in E$
so 
\begin{eqnarray*}
{\rm dist}(\xi, E) &\leq& \|\xi - R^{-1}\eta S^{-1}\| =
\|R\xi S - \eta\| \\
&\leq& \sqrt 2 (|\lambda_1(\xi) - a| + |\lambda_2(\xi) - b|) \leq 3\sqrt 2 \varepsilon_n.
\end{eqnarray*}
Thus we conclude that
$$\sup\limits_{\xi \in U_n}{\rm dist}(\xi, E) \to 0 \; \; {\rm as}
\; \; n \to +\infty.$$
\end{proof}

In fact, given a compact set $E$ satisfying the hypothesis of the previous
theorem, it is possible to obtain an in-approximation sequence for $E$ such that
if $\xi \in {\rm int} E^{rc}$ then $\xi$ belongs to the first set of this sequence.
In order to prove this stronger result, which will be used in the next section, we will need 
the following definitions.

\begin{defn}\label{newsets} 
Let $\varepsilon_n$ be the sequence in Definition~\ref{defcasosoloab}
and let $\delta_n$ be such that $0 < \delta_n < \frac{\varepsilon_n}{4}$. 
For $(a,b) \in K$ we define the compact sets,
\[
\begin{array}{l}
\tilde{R}^n_{(a,b)}=
\{(x,y)\in \Real^2: a-2\varepsilon_{n}+\delta_n 
\leq x \leq a-\varepsilon_{n} - \delta_n,
\\
\phantom {R^n_{(a,b)}= \{ (x,y)\in T:} \,\,\displaystyle
b-\varepsilon_n + \delta_n \leq y \leq b -\frac{\varepsilon_{n}}{2} - \delta_n\}
\end{array}
\]
and    
$\displaystyle 
C_n:=\left\{\xi \in \Real^{2\times 2}: (\lambda_1(\xi),\lambda_2(\xi))\in
K_n\right\}$, where 
$$\displaystyle K_n=\bigcup\limits_{(a,b)\in K}\tilde{R}^n_{(a,b)}.$$
\end{defn}

\begin{rem}{\rm Notice that, given the choice of $\varepsilon_n$ and $\delta_n$, 
the sets $\tilde{R}^n_{(a,b)}$ are non-empty and lie in the 
region $\{(x,y) \in \Real^2 : 0 < x \leq y\}$.}
\end{rem}

\begin{prop}\label{propftheta} 
The function $f_\theta(x,y)$ defined in (\ref{ftheta}) satisfies the following properties:
\begin{itemize}

\item[i)] $f_\theta$ is strictly increasing in $y$, for every $x > 0$ and 
$\theta \geq 0$;

\item[ii)] $f_\theta$ is strictly increasing in $x$, for every $y > \theta$ and 
is strictly decreasing in $x$, for every $y < \theta$;

\item[iii)] $f_\theta(\cdot,\theta)$ is constant, for every $\theta \geq 0$;

\item[iv)] setting 
$$
\begin{array}{lll}
\alpha^n_{(a,b)}(\theta) & = & f_\theta\left(a-2\varepsilon_n
+\delta_n,b-\frac{\varepsilon_n}{2}-\delta_n\right)
\\
\beta^n_{(a,b)}(\theta) & = &
f_\theta\left(a-\varepsilon_n-\delta_n,b-\frac{\varepsilon_n}{2}-\delta_n\right)
\end{array}
$$
with $\varepsilon_n$ and $\delta_n$ as in Definitions \ref{defcasosoloab} and \ref{newsets}, 
one has
$$
\max\limits_{(x,y)\in \tilde{R}^{n}_{(a,b)}}f_\theta(x,y)
= \max\left\{\alpha^n_{(a,b)}(\theta),\beta^n_{(a,b)}(\theta)\right\}
$$
$$
=\left\{
\begin{array}{l}
\beta^n_{(a,b)}(\theta), \, {\rm if} \, \theta \in [0,\max\limits_{(x,y)\in
\tilde{R}^n_{(a,b)}}y]
\\
\alpha^n_{(a,b)}(\theta), \, {\rm if} \, \theta \geq \max\limits_{(x,y)\in
\tilde{R}^n_{(a,b)}}y\,.
\end{array}\right.$$ 
\end{itemize}
\end{prop}
\begin{proof}
The first three properties are clear and the fourth one follows from $i)$, 
$ii)$ and $iii)$.
\end{proof}

We are now in position to prove the stronger in-approximation property which is
required to obtain our existence result.

\begin{thm}\label{datoalbordo} Let $E$ be the set defined by
(\ref{notreE}) and (\ref{Knumerofinitodipunti}) and let 
$\xi \in \textnormal{int}E^{rc}$. Then there exists an in-approximation sequence 
$U_n$ for $E$ such that $\xi \in U_1$.
\end{thm}
\begin{proof}
Consider the sequence of sets $U_n$ given in Definition~\ref{defcasosoloab}
and let $\xi \in \textnormal{int}\,E^{rc}$.
We will show that there exists $N=N(\xi) \in \N$ such that 
\begin{equation}\label{limitelemme}
\xi \in (U_N)^{rc}.
\end{equation}
Given that $U_N$ is open, it suffices to show that $\xi$ belongs to $C^{rc}$ 
for a certain compact subset $C$ of $U_N$.

Let $C_n$ be the sets defined in Definition~\ref{newsets}. Clearly 
$C_n \subseteq U_n$
and $C_n$ are bounded since $\lambda_2$ is a norm and $K$ is compact.
To prove that each $C_n$ is closed we consider a sequence $\xi_m \in C_n$ such that
$\xi_m \to \xi$ as $m \to + \infty$. Then, for every $m$ there exists 
$(a_m,b_m) \in K$
such that $(\lambda_1(\xi_m),\lambda_2(\xi_m)) \in \tilde{R}^n_{(a_m,b_m)}$.
As $K$ is compact, there exists $(a,b) \in K$ such that, up to a subsequence,
$(a_m,b_m) \to (a,b)$ as $m \to +\infty$. By the inequalities that define
$\tilde{R}^n_{(a_m,b_m)}$ and the continuity of $\lambda_1$ and $\lambda_2$ we 
conclude that
$(\lambda_1(\xi),\lambda_2(\xi)) \in \tilde{R}^n_{(a,b)}$ and thus $\xi \in C_n$.
Thus $C_n$ is compact.

In order to choose an appropriate $N$ to satisfy (\ref{limitelemme}) we begin by showing that
\begin{equation}\label{limite}
\max\limits_{(a,b)\in K_n}f_\theta (a,b) \to \max\limits_{(a,b)\in K}f_\theta (a,b), 
\,\,\,\,\,\,\,n\to +\infty,
\end{equation}
uniformly with respect to $\theta \in [0,\max\limits_{(a,b)\in K}b]$. 
By Proposition~\ref{propftheta}, $iv)$
$$
\max\limits_{(a,b) \in K_n}f_\theta(a,b) =
\sup\limits_{(a,b)\in K} \max\{\alpha^n_{(a,b)}(\theta),
\beta^n_{(a,b)}(\theta)\}
$$
and
$$
|\sup\limits_{(a,b)\in K}
\max\{\alpha^n_{(a,b)}(\theta),\beta^n_{(a,b)}(\theta)\}-
\max\limits_{(a,b) \in K}f_\theta(a,b)|
$$
$$
\leq \sup\limits_{(a,b)\in K}
|\max\{\alpha^n_{(a,b)}(\theta),
\beta^n_{(a,b)}(\theta)\}-f_\theta(a,b)|.
$$
Therefore we must show that, as $n \to +\infty$,  
$$
|\alpha^n_{(a,b)}(\theta)-f_\theta(a,b)| \to 0
\,,\,\,
|\beta^n_{(a,b)}(\theta)-f_\theta(a,b)| \to 0,
$$
uniformly with respect to $\theta$ and to $(a,b)$. We start with the first 
limit. Letting
$$
m_n=b-\frac{\varepsilon_n}{2}-2\delta_n-a+2\varepsilon_n,
$$
and
$$
q_n=\left(a-2\varepsilon_n
+\delta_n\right)\left(b-\frac{\varepsilon_n}{2}-\delta_n\right)
$$
we have $\alpha^n_{(a,b)}(\theta)=m_n\theta+q_n$. Notice that 
$q_n-ab\to 0$ and $m_n-b+a\to 0$ uniformly with respect to $(a,b)$. 
This implies the result.
The same reasoning applies to the second limit.

As $\xi \in \textnormal{int}\,E^{rc}$ we know that
$$f_\theta(\lambda_1(\xi),\lambda_2(\xi)) < \max\limits_{(a,b)\in K}f_\theta (a,b),
\,\,\forall\,\, \theta\in [0,\max\limits_{(a,b)\in K}b]
$$
thus, there exists $\tau > 0$ such that
$$f_\theta(\lambda_1(\xi),\lambda_2(\xi)) 
< \max\limits_{(a,b)\in K}f_\theta (a,b) - \tau,
\,\,\forall\,\, \theta\in [0,\max\limits_{(a,b)\in K}b].
$$
By the uniform convergence shown in (\ref{limite}), for this $\tau$ there exists 
$N \in \N$
such that
$$\max\limits_{(a,b)\in K}f_\theta (a,b) -\tau 
< \max\limits_{(a,b)\in K_N}f_\theta (a,b),
\,\,\forall\,\, \theta\in [0,\max\limits_{(a,b)\in K}b]
$$
and thus 
$$f_\theta(\lambda_1(\xi),\lambda_2(\xi)) 
< \max\limits_{(a,b)\in K_N}f_\theta (a,b),
\,\,\forall\,\, \theta\in [0,\max\limits_{(a,b)\in K}b].
$$
In particular, as $\max\limits_{(a,b)\in K_N}b < \max\limits_{(a,b)\in K}b$, 
we obtain
$$f_\theta(\lambda_1(\xi),\lambda_2(\xi)) < \max\limits_{(a,b)\in K_N}f_\theta (a,b),
\,\,\forall\,\, \theta\in [0,\max\limits_{(a,b)\in K_N}b]
$$
which proves that $\xi \in (C_N)^{rc}$.
As $C_N$ is a compact subset of $U_N$ we have, therefore, shown (\ref{limitelemme}).

To complete the proof we notice that the sequence
$$(U_N)^{rc}, U_{N+1}, U_{N+2},...
$$
is an in-approximation of $E$. Indeed, as $U_N$ is open $(U_N)^{rc}$ is also open so all
the sets of the above sequence are open.
Since, by construction (cf. Theorem~\ref{inapproximationteoexistence}), 
$U_N \subseteq (U_{N+1})^{rc}$ we conclude that
$(U_{N})^{rc} \subseteq ((U_{N+1})^{rc})^{rc} = (U_{N+1})^{rc}$,
by Theorem~\ref{Erc=Epc}, (ii).
Moreover, if $\xi \in \textnormal{int} E^{rc}$ then
$\xi \in (U_N)^{rc}$, by
(\ref{limitelemme}). 
\end{proof}

\section{Existence Theorem}
In this section we will prove Theorem 
\ref{teorema d'esistenza nel nostro caso sverakmuller}.
We will assume that the boundary datum $\varphi$ is 
$C^1_{piec}(\overline{\Omega}, \Real^2)$, 
that is to say, $\varphi \in W^{1,\infty}({\Omega}, \Real^2)$, there exist open sets
$\omega_i\subset \Omega$ such that 
$\varphi \in C^1(\overline{\omega_i},\Real^2)$
and $\displaystyle \Omega\setminus \bigcup\limits_{i}\omega_i$ is a set of 
Lebesgue measure zero.

We recall the notion of fine $C^0$-approximation which can be found in \cite{MS}.

\begin{defn}
Let $\mathcal{F}(\Omega,\Real^N)$ be a family of continuous mappings of $\Omega$ into
$\Real^N$. We say that a given continuous mapping $v_0 : \Omega \to \Real^N$ admits
a fine $C^0$-approximation by the family $\mathcal{F}(\Omega,\Real^N)$ if there
exists, for every continuous function $\varepsilon : \Omega \to (0,+\infty)$, an
element $v$ of the family $\mathcal{F}(\Omega,\Real^N)$ such that
$|v(x) - v_0(x)| < \varepsilon(x)$ for each $x \in \Omega$.
\end{defn}

The proof of Theorem~\ref{teorema d'esistenza nel nostro caso sverakmuller}
is based on the following abstract theorem of M\"uller and \v{S}ver\'ak \cite{MS}.

\begin{thm}\label{existence sverak-muller}
Let $\Omega \subset \Real^n$ be an open, bounded set and let 
$E \subseteq \Real^{N\times n}$ be a compact
set which admits an in-approximation by the open sets $U_i$. Let 
$\varphi:\Omega \to \Real^N$ be a 
$C^1$ function such that $D \varphi \in U_1$. Then 
$\varphi$ admits a fine $C^0$-approximation by Lipschitz mappings
$u:\Omega \to \Real^N$
such that $D u \in E$ a.e. in $\Omega.$
\end{thm}

In the following result we begin by considering the case where $\varphi$ is an 
affine function.

\begin{thm}\label{affine}
Let $\Omega$ be an open, bounded subset of $\Real^2$ and let $E$ be the set
defined by (\ref{notreE}) and (\ref{Knumerofinitodipunti}). Let
$\xi \in \Real^{2\times 2}$ be such that 
$\xi \in \textnormal{int}\,E^{rc}$ and let $\varphi: \Omega\to \Real^2$ satisfy
$D \varphi=\xi$ in $\Omega$. Then there exists 
$u\in \varphi+W^{1,\infty}_0(\Omega,\Real^2)$ such that $Du \in E$ a.e. in $\Omega$.
\end{thm}
\begin{proof}
Given $\xi \in \textnormal{int}\,E^{rc}$ by Theorem~\ref{datoalbordo}
there exists an in-appoximation sequence $U_n$ for $E$ such that $\xi \in U_1$.
Thus, if $\varphi: \Omega\to \Real^2$ is a $C^1$ mapping such that 
$D \varphi=\xi$, by Theorem~\ref{existence sverak-muller},
$\varphi$ admits a fine $C^0$-approximation by Lipschitz mappings
$u:\Omega \to \Real^2$ satisfying $D u \in E$ a.e. in $\Omega$.
Hence for every continuous function $\varepsilon : \Omega \to (0,+\infty)$
there exists $u \in W^{1,\infty}({\Omega}, \Real^2)$ such that $D u \in E$ a.e. in
$\Omega$ and
$$|u(x) - \varphi(x)| < \varepsilon(x), \; \; \forall \, x \in \Omega.$$
Let
$\varepsilon : \overline{\Omega} \to [0,+\infty)$ be a continuous function such
that $\varepsilon(x) = 0 \Leftrightarrow x \in \partial \Omega$ and
extend $\varphi$ as a $C^1$ mapping to $\overline{\Omega}$.
Let 
$x_0 \in \partial \Omega$ and $x_n \in \Omega$ be a sequence such that
$x_n \to x_0$. Passing to the limit the inequality
$$|u(x_n) - \varphi(x_n)| < \varepsilon(x_n)$$
we conclude that $u \in \varphi + W^{1,\infty}_0({\Omega}, \Real^2)$.
\end{proof}

To obtain our existence result in the general case we will once again make use of 
Theorem~\ref{datoalbordo} together with the following result, proved by
Dacorogna and Marcellini in \cite{DM} (Corollary 10.15).

\begin{thm} Let $\Omega$ be an open subset of $\Real^n$ and $A$ be an open subset of
$\Real^{N \times n}$. Let 
$\varphi \in C^1(\Omega, \Real^N)\cap W^{1,\infty}(\Omega, \Real^N)$ be such that
$$D\varphi(x)\in A, \, \forall\,\, x \in \Omega.$$
Then there exists a function $v\in W^{1,\infty}(\Omega, \Real^N)$
such that $v$ is piecewise affine in $\Omega$, $v=\varphi$ on $\partial \Omega$
and $D v\in A$ a.e. in $\Omega.$
\end{thm}

We now proceed with the proof of 
Theorem~\ref{teorema d'esistenza nel nostro caso sverakmuller}.

\begin{proof}
Assume first that $\varphi\in C^1(\overline{\Omega},\Real^2)$. 
Let
$$\Omega_0 = \left\{x \in \Omega : \, D\varphi(x) \in E\right\}$$
and $\Omega_1 = \Omega \setminus \Omega_0$. 
Since $E$ is closed and $\varphi$ is $C^1$, the set $\Omega_0$ is closed
and thus $\Omega_1$ is open.
In $\Omega_1$ we apply the previous theorem to $\varphi$ and to the open set
$\textnormal{int}\,E^{rc}$
in order to obtain a map $v\in W^{1,\infty}(\Omega_1,\Real^2)$
such that $v=\varphi$ on $\partial \Omega_1$, 
$D v =c_i$ in $\Omega_1^i$ for some constant $c_i \in \textnormal{int}\,E^{rc}$ 
and $\bigcup\limits_i\Omega_1^i=\Omega_1$. 
Due to Theorem~\ref{affine} we can solve the problem
$$\left\{
\begin{array}{ll}
D u \in E, \, &\mbox{a.e. in } \Omega_1^i \vspace{0.1cm}
\\
u(x)=v(x), &x \in \partial \Omega_1^i 
\end{array}
\right.
$$ 
in each set $\Omega_1^i$. Denoting by $u_i$ the solution in
$\Omega_1^i$, the map defined by 
$$u = \left\{
\begin{array}{ll}
u_i, \, &\mbox{ in } \Omega_1^i \vspace{0.1cm}
\\
\varphi, &\mbox{ in } \Omega_0 
\end{array}
\right.
$$ 
belongs to $\varphi+W_0^{1,\infty}(\Omega,\Real^2)$ and
satisfies $Du \in E$.

Now suppose that $\varphi\in C^1_{piec}(\overline{\Omega},\Real^2)$.
This means that there exist open sets
$\omega_i\subset \Omega$ such that 
$\varphi \in C^1(\overline{\omega_i},\Real^2)$
and $\displaystyle \Omega\setminus \bigcup\limits_{i}\omega_i$ is a set of Lebesgue 
measure zero.
By the first case, for each $i$, there exists 
$w_i \in \varphi+W_0^{1,\infty}(\omega_i,\Real^2)$ such that 
$Dw_i \in E$ a.e. in $\omega_i$.
Thus, the function $u$ defined as $w_i$ in $\omega_i$ belongs to 
$\varphi+W_0^{1,\infty}(\Omega,\Real^2)$ and satisfies $Du \in E$,
a.e. in $\Omega$.
\end{proof}

We conclude this article by pointing out that 
Theorem~{\ref{teorema d'esistenza nel nostro caso sverakmuller}} 
is not far from being optimal 
in the case where the boundary datum $\varphi$ is affine.

Indeed, suppose that $u$ is a solution of
$$
\left\{
\begin{array}{ll}
Du \in E,& \text{a.e.  in} \,\Omega
\\
u=u_{\xi_0},& \text{on } \partial\Omega
\end{array}
\right.
$$
where $u_{\xi_0}$ is an affine function with $Du_{\xi_0} = \xi_0$. 
Then there exists a map 
$\psi \in W^{1,\infty}_0(\Omega,\Real^2)$ such that $u=u_{\xi_0}+\psi$.
Let $f \in \mathcal{F}^E$ be a polyconvex function. Then $f$ is also 
quasiconvex and thus 
$$
f(\xi_0)\leq \frac{1}{|\Omega|}\int_{\Omega}f(\xi_0+D\psi) \, dx 
=\frac{1}{|\Omega|}\int_{\Omega}f(Du)\, dx  \leq 0
$$
since $f\lfloor_{E}\leq 0$. This implies that $\xi_0 \in E^{pc}$.
As already mentioned at the end of Section 3, in the case where $E$ is an 
isotropic compact subset of 
$\Real^{2\times 2}$, results of \cite{C} and \cite{CT}
imply that $\text{Rco}\,E= E^{rc}= E^{pc}$.
Therefore $\xi_0 \in E^{rc}$.

\section*{Acknowledgements}
The research of Ana Cristina Barroso was partially supported by
Funda\c c\~ao para a Ci\^encia e Tecnologia, Financiamento Base
2009-ISFL/1/209.
Ana Margarida Ribeiro was partially supported by
Financiamento Base 2009-ISFL/1/297 from FCT/MCTES/PT.

%\footnotesize

% Bibliography----------------------------------------------------------------

\end{document}